\documentclass{article}
\usepackage{amsmath}
\usepackage{amssymb}
\usepackage{amsthm}
\usepackage{mathtools}
\usepackage{nicematrix}
\usepackage[colorlinks,citecolor=red,linkcolor=blue]{hyperref}
\usepackage{titlesec}
\usepackage{tikz-cd}
\usepackage{geometry}
\usepackage{amsmath, amssymb, graphics, setspace}

\newtheorem{theorem}{Theorem}[section]
\newtheorem{proposition}[theorem]{Proposition}

\newtheorem{lemma}[theorem]{Lemma}

\newtheorem{corollary}[theorem]{Corollary}

\theoremstyle{definition}
\newtheorem{definition}[theorem]{Definition}

\newtheorem{eg}[theorem]{Example}
\newtheorem{remark}[theorem]{Remark}

\title{Hodge structures of the surface of planes in a cubic $5$-fold}
\author{Chenpeng Feng}
\date{}

\begin{document}

\maketitle

\begin{abstract}
    We study the geometry of the moduli space of planes in a general cubic $5$-fold and its deformation. We show that this moduli space is a smooth projective surface whose canonical bundle is ample. We also show that the variation of degree 1 Hodge structures of a particular family of such surfaces is maximal. The main technical input is the hyper-Kähler geometry and an elaborated calculation of the Hodge numbers of such surfaces
\end{abstract}
\tableofcontents
\section{Introduction}

The moduli space of linear subspaces in a projective hypersurfaces has been extensively studied in the litterature~\cite{Bai, BD, DM, KCorr}, starting from the classical result of Cayley and Salmon which states that every smooth projective cubic surface has exactly $27$ lines. In the language of moduli spaces, it says that the moduli space of lines in a cubic surface is a finite set of 27 points.. Another important example is the moduli space of lines in a cubic $4$-fold. As is proven in~\cite{BD}, this moduli space is a hyper-Kähler manifold, a fundamental building block in the classification of algebraic varieties.

In this article, we consider the following explicit moduli space of such type. Let $Y\subset \mathbb P^6$ be a general cubic $5$-fold. Let $\Sigma$ be the moduli space of planes in $Y$. As is proved in Section~\ref{SectionSurface}, $\Sigma$ is a smooth projective surface of general type. The study of the surface $\Sigma$ is motivated by a result of Iliev and Manivel~\cite[Proposition 4]{IlievManivel}, which states that the surface $\Sigma$ is a Lagrangian subvariety of the Fano variety of lines of a cubic $4$-fold $X$ if the cubic $5$-fold is chosen general.

Our first result is the calculation of the Betti numbers of the smooth surface $\Sigma$.
\begin{theorem}[Theorem~\ref{ThmBettiNumbers}]\label{ThmBettiNumbersIntro}
    The Betti numbers of the surface $\Sigma$ are as follows.

 (i) $b_0(\Sigma) = b_4(\Sigma) = 1$. 
 
    (ii) $b_1(\Sigma) = b_3(\Sigma) = 42$. 
    
    (iii) $b_2(\Sigma) = 13123$.
\end{theorem}

Since $\Sigma$ is a smooth projective surface, its cohomology groups admits a Hodge decomposition (Section~\ref{SubsectionHodgeDecomp}). Our second result is the calculation of the Hodge numbers.
\begin{theorem}[Theorem~\ref{ThmHodgeNumbers}]\label{ThmHodgeNumbersIntro}
    The Hodge numbers of the surface $\Sigma$ are as follows. 

(i) $h^{2, 0}(\Sigma) = h^{0, 2}(\Sigma) =3233.$

(ii) $h^{1, 1}(\Sigma) = 6657$.

(iii) $h^{1,0}(\Sigma)=h^{0, 1}(\Sigma)=21$.
\end{theorem}

Our third result is the study of the variation of Hodge structures of $\Sigma$. Let $B$ be the moduli of cubic $5$-folds $Y\subset \mathbb P^6$ containing a given cubic $4$-fold $Y_4$ inside $\mathbb P^6$. For each element $b\in B$, consider the moduli of planes in the corresponding cubic $5$-fold $Y_b$, we get a family of surfaces parametrized by an open subset of $B$

\begin{theorem}\label{ThmVHSisMaximal}
    The variation of degree $1$ Hodge structures of the above family is maximal. 
\end{theorem}

 Here, the variation of Hodge structures is called maximal, if the period map is locally an immersion (i.e., the tangent map is generally injective). See Section~\ref{SectionVHS} for detailed definitions.

Let us briefly present the method for the above results. Notice that Theorem~\ref{ThmBettiNumbersIntro} and Theorem~\ref{ThmHodgeNumbersIntro} are also obtained independently in \cite{Mboro} during the preparation of this work, using different methods. Theorem~\ref{ThmVHSisMaximal} is new.

As $\Sigma$ parametrizes the planes in a general cubic $5$-fold $Y\subset \mathbb P^6$, we may regard $\Sigma$ as a subvariety of the Grassmannian $\mathrm{Gr}(3, 7)$ defined as the zero locus of a general section of the vector bundle $\mathrm{Sym}^3\mathcal U^\vee$, where $\mathcal U$ is the tautological subbundle of $Gr(3, 7)$. Since $\mathrm{Gr}(3, 7)$ is of dimension $12$ and $\mathrm{Sym}^3\mathcal U^\vee$ is of rank $10$, the dimension of $\Sigma$ is $2$. Therefore, the tangent bundle $T_\Sigma$ fits into the following short exact sequence
\begin{equation}\label{EqExactSequenceForTangent}
    0 \to T_\Sigma\to T_{Gr(3, 7)}|_\Sigma \to \mathrm{Sym}^3\mathcal U^\vee|_\Sigma \to 0,
\end{equation}
which gives an expression of the canonical bundle $K_\Sigma$ of $\Sigma$ as follows
\[
K_\Sigma = (K_{Gr(3, 7)}\otimes \det \mathrm{Sym}^3\mathcal U^\vee)|_\Sigma.
\]
The right-hand side can be easily seen to be ample (Proposition 2.4).

To calculate the Betti numbers of surface $\Sigma$, we first calculate the Euler characteristic of $\Sigma$ (note that $\Sigma$ is a compact manifold of \emph{real} dimension $4$):
\[
e(\Sigma) = \sum_{i = 0}^4 (-1)^ib_i(\Sigma).
\]
By the Hopf-Poincaré theorem, $e(\Sigma)$ is equal to the top degree Chern class $c_2(\Sigma)$ of the tangent bundle of $T_\Sigma$, viewed as an integer via the natural identification $\int_\Sigma: H^4(\Sigma, \mathbb Z)\cong\mathbb Z$. Basic theory of Chern classes is briefly recalled in Section~\ref{SubsectionChernClasses}. By the exact sequence (\ref{EqExactSequenceForTangent}), we can calculate the Chern class $c_2(\Sigma)$ as follows
$$c_2(\Sigma) = \left( c_1(\mathrm{Sym}^3 U^\vee)^2 - c_1(U^\vee \otimes Q) \cdot c_1(\mathrm{Sym}^3 U^\vee) + c_2(U^\vee \otimes Q) - c_2(\mathrm{Sym}^3 U^\vee) \right) \cdot c_{10}(\mathrm{Sym}^3 U^\vee).$$
By the Schubert calculus that we will briefly recall in Section~\ref{SubsectionSchubert}, the right-hand-side can be calculated by combinatorial methods (see Section~\ref{SectionChernCalculations} for explicit calculations) and the result is $13041$. Now that we know the Euler characteristic, to calculate all the Betti numbers, it suffices to know either $b_1$ or $b_2$. However, a result of Collino~\cite{Collino} shows that $b_1(\Sigma) = 42$. This gives all the Betti numbers of $\Sigma$.

Next, to calculate the Hodge numbers of $\Sigma$, we use the Noether's formula
\[
\chi(\mathcal{O}_\Sigma) = \frac{1}{12}(c_1(\Sigma)^2 + c_2(\Sigma)).
\]
As the calculation of $c_2(\Sigma)$ illustrated in the previous paragraph, $c_1(\Sigma)^2$ can be calculated by using Schubert calculus and the result is $25515$. Hence, $\chi(\Sigma) = 3213$. By the result of Collino again, we get $h^{1,0}(\Sigma) = 21$. Thus, $h^{2, 0}(\Sigma) = 3233$. Combined with the calculations of the Betti numbers, we finally obtain $h^{1,1}(\Sigma) = 6657$. All Hodge numbers are thus calculated.

Finally, let us discuss the result about the variation of Hodge structures of $\Sigma$. Although this result is only related to the family of the surfaces $\Sigma$, our strategy consists of viewing this family as a Lagrangian family of the Beauville-Donagi hyper-Kähler $4$-fold $X$~\cite{IlievManivel, BD}. Then we use a recent result on the Lagrangian families in the hyper-Kähler geometry. By a result of Bai~\cite[Proposition 2]{Bai}, if, in general, a Lagrangian family satisfies some numerical conditions on the Hodge numbers, then the variation of degree $1$ Hodge structures of the Lagrangian family is maximal. The above calculation on the Hodge numbers of the surface $\Sigma$ confirms the numerical conditions needed, and thus the family as in Theorem~\ref{ThmVHSisMaximal}, which is a Lagrangian family in the Beauville-Donagi hyper-Kähler $4$-fold, has a maximal variation of degree $1$ Hodge structures.

\section{Hyper-Kähler manifolds and Lagrangian submanifolds}\label{SectionSurface}

\subsection{Kähler Manifolds and Hodge Numbers}\label{SubsectionHodgeDecomp}
The best way to discuss Hodge theory is through Kähler geometry. In this section, we follow closely the presentation of \cite[Chapter 6]{VoisinHodge} to present the Hodge theory

\begin{definition}
A \textbf{Kähler manifold} is a complex manifold $X$ equipped with a Hermitian metric whose associated $2$-form is closed.
\end{definition}

\begin{definition}
For a Kähler manifold $X$, the \textbf{Dolbeault cohomology groups} $H^{p,q}(X)$ consist of equivalence classes of differential forms of type $(p,q)$, modulo $\overline{\partial}$-exact forms:
$$H^{p,q}(X) = \frac{\ker \overline{\partial}: A^{p,q}(X) \to A^{p,q+1}(X)}{\text{im } \overline{\partial}: A^{p,q-1}(X) \to A^{p,q}(X)},$$
where $A^{p,q}(X)$ denotes the space of smooth differential forms of type $(p,q)$.
\end{definition}

\begin{theorem}[Hodge decomposition (\cite{VoisinHodge})]
    Let $X$ be a compact Kähler manifold. Then there is a canonical decomposition of the de Rham cohomology of $X$ as follows.
    \[
    H^k(X,\mathbb C) = \bigoplus_{p + q = k} H^{p, q}(X),
    \]
    with the complex conjugate identification $H^{p, q}(X) = \overline{H^{q, p}(X)}$.
    Furthermore, we have a canonical isomorphism $H^{p, q}(X) \cong H^q(X, \Omega_X^p)$.
\end{theorem}

\begin{definition}\label{DefHodgeNumbers}
The \textbf{Hodge numbers} of a compact Kähler manifold $X$ are defined as the dimensions of the Dolbeault cohomology groups:
$$h^{p,q}(X) = \dim H^{p,q}(X).$$
\end{definition}

The Hodge numbers $h^{p,q}(X)$ satisfy several symmetries, notably the following:
\begin{itemize}
    \item \textbf{Hodge Symmetry}: $h^{p,q}(X) = h^{q,p}(X)$.
    \item \textbf{Serre Duality}: $h^{p,q}(X) = h^{n-p,n-q}(X)$ for $n = \dim_\mathbb{C} X$.
\end{itemize}

For a compact Kähler manifold \( X \) of complex dimension \( n \), the Hodge diamond (which organizes the Hodge numbers \( h^{p,q}(X) \)) is structured as follows:

\[
\begin{array}{ccccccccccc}
 & & & & & h^{0,0} & & & & & \\
 & & & & h^{1,0} & & h^{0,1} & & & & \\
 & & & h^{2,0} & & h^{1,1} & & h^{0,2} & & & \\
 & & \iddots & & & \vdots & & & \ddots & & \\
 & h^{n-1,0} & & h^{n-2,1} & & \cdots & & h^{1,n-2} & & h^{0,n-1} & \\
 h^{n,0} & & h^{n-1,1} & & h^{n-2,2} & & \cdots & & h^{1,n-1} & & h^{0,n} \\
 & h^{n,n-1} & & h^{n-1,n-2} & & \cdots & & h^{2,1} & & h^{n,0} & \\
 & & \ddots & & & \vdots & & & \iddots & & \\
 & & & h^{n-2,n} & & h^{n-1,n-1} & & h^{n,n-2} & & & \\
 & & & & h^{n,n-1} & & h^{n-1,n} & & & & \\
 & & & & & h^{n,n} & & & & & \\
\end{array}
\]

It is not hard to show that any smooth projective variety over $\mathbb C$ is a compact Kähler manifold, and hence the whole Hodge theory applies. The key takeaway here is that the Hodge numbers can be calculated as the dimension of the cohomology of the bundles of holomorphic differential forms, whereas the latter has the potential to be fitted into exact sequences. That is exactly the approach we take for the calculations of Hodge numbers of the smooth projective surface $\Sigma$.





\subsection{Variation of Hodge structures}\label{SectionVHS}

Let $\pi: \mathcal X\to B$ be a holomorphic family of compact Kähler manifolds. Up to restricting $B$ to a smaller open polydisc, we may assume that $B$ is simply connected. By the Ehresmann's theorem, the fibers of the family are diffeomorphic to each other in a canonical way, thus the Betti cohomology groups of the fibers can be canonically identified. 
\begin{definition}
The \textbf{variation of Hodge structures} of weight \( k \) of the family $\pi: \mathcal X\to B$ consists of:
\begin{enumerate}
    \item A fixed abelian group $H_\mathbb Z$, which is a representative of the $k$-th Betti cohomology group mentioned above;
    \item and for each $s\in B$, the Hodge structure $H^{p, q}(X_s)$ with ambiant space $H^k(X_s, \mathbb Z)\cong H_\mathbb Z$.
\end{enumerate}
\end{definition}

The variation of Hodge structures encodes how the Hodge decomposition \( H_\mathbb{C} = \bigoplus_{p+q=k} H^{p,q} \) changes as the parameter \( s \) varies over \( B \).

To study variations of Hodge structures, it is useful to consider the period mapping. This is a holomorphic map that associates to each point \( s \in B \) the Hodge structure on the fiber \( H_s \). For each $s \in B$, the subspace $H^{p, q}(X_s)$ is a subspace of the fixed vector space $H_\mathbb C = H_\mathbb Z\otimes_\mathbb Z \mathbb C$, which can viewed as an element of the Grassmannian $\mathrm{Gr}(h^{p, q}, H_\mathbb C)$.

\begin{definition}
The \textbf{period mapping} \( \Phi \) is a holomorphic map:
\[
\Phi: B \to \prod_{p+ q = k} \mathrm{Gr}(h^{p, q}, H_\mathbb C),
\]
as described above.
\end{definition}

\begin{definition}[Maximal Variation]
We say that the variation of the Hodge structure of degree \( k \) is \textbf{maximal} if the period mapping:
\[
\Phi: B \to \prod_{p+q=k} \mathrm{Gr}(h^{p,q}, H_\mathbb C)
\]
is an immersion, i.e., the derivative of \( \Phi \) is injective at every point \( s \in B \).
\end{definition}

This condition means that the Hodge structure varies as freely as possible across the family, with no redundant dependencies among the parameters.


\subsection{Hyper-Kähler manifolds}

We follow closely the presentation in~\cite{DebarreHK} for the hyper-Kähler manifold.

\begin{definition}
    A hyper-Kähler manifold $X$ is a simply connected, compact Kähler manifold such that $H^0(X, \Omega_X^2) = \mathbb C\sigma_X$, where $\sigma_X$ is a nowhere degenerate holomorphic $2$-form.
\end{definition}


\begin{eg}[Beauville-Donagi Hyper-Kähler Fourfold~\cite{BD}]
    Let \( Y_4 \subset \mathbb{P}^5 \) be a smooth cubic fourfold, and let \( X = F_1(Y_4) \) denote the Fano variety of lines on \( Y_4 \), defined as:
    \[
    F_1(Y_4) = \{ \ell \subset Y_4 \mid \ell \text{ is a line} \}.
    \]
    In~\cite{BD}, it is shown that $X$ is a hyper-Kähler manifold of dimension $4$.
    The Betti numbers of the hyper-Kähler $4$-fold $X$ are as follows:
Here is the revised formula with \( Y_4 \) changed to \( X \):

\[
b_0(X) = 1, \quad b_1(X) = 0, \quad b_2(X) = 23, \quad b_3(X) = 0, \quad b_4(X) = 276,
\quad b_5(X) = 0, \quad b_6(X) = 23, \quad b_7(X) = 0, \quad b_8(X) = 1
\]

The Hodge numbers of \( Y_4 \) are:
\[
h^{0,0} = h^{4,4} = 1, \quad h^{1,1} = 21, \quad h^{2,2} = 232, \quad h^{3,3} = 21, \quad h^{1,3} = h^{3,1} = 1.
\]
\end{eg}

\subsection{Transcendental Parts and Hodge Structures of Hyper-Kähler Manifolds}\label{SectionTranscendental}

In this section, we study the transcendental part of the cohomology of hyper-Kähler manifolds and its relationship to cubic fourfolds. These concepts are central to understanding the geometric and arithmetic properties of these varieties.

\begin{definition}
Let \( X \) be a hyper-Kähler manifold. The \textbf{transcendental part} of \( H^2(X, \mathbb{Q}) \) is defined to be the minimal Hodge structure containing \( H^{0,2}(X) \). It is denoted as:
\[
H^2(X, \mathbb{Q})_{\mathrm{tr}} \subset H^2(X, \mathbb{Q}),
\]
where \( H^2(X, \mathbb{Q})_{\mathrm{tr}} \) is the transcendental part of \( H^2 \).
\end{definition}

\begin{remark}
For a projective hyper-Kähler manifold \( X \), the transcendental part can be expressed as the orthogonal complement of the Néron-Severi group in \( H^2(X, \mathbb{Q}) \) under the Beauville-Bogomolov quadratic form:
\[
H^2(X, \mathbb{Q})_{\mathrm{tr}} = \left[ H^2(X, \mathbb{Q}) \cap H^{1,1}(X) \right]^\perp,
\]
where \( \mathrm{NS}(X) \) denotes the Néron-Severi group.
\end{remark}

\begin{theorem}[Bolognesi-Pedrini~\cite{Bolognesi}]
Let \( Y \subset \mathbb{P}^5 \) be a cubic fourfold, and let \( X = F_1(Y) \) denote the Fano variety of lines in \( Y \). Then:
\[
H^2(X, \mathbb{Q})_{\mathrm{tr}} \cong H^4(Y, \mathbb{Q})_{\mathrm{tr}}.
\]
\end{theorem}
The theorem below is a classical result whose proof can be found in, for example, ~\cite[IV, Corollary 3.3]{Hartshorne}.
\begin{theorem}
For a very general cubic fourfold \( Y \), the Néron-Severi group \( \mathrm{NS}(Y)_\mathbb{Q} \) is \( \mathbb{Q} \).
\end{theorem}

\begin{corollary}\label{CorTranscendental}
For a very general cubic fourfold \( Y \):
\begin{enumerate}
    \item \( b_4(Y)_{\mathrm{tr}} = b_4(Y) - 1 \), 
    \item \( b_2(X)_{\mathrm{tr}} = b_4(Y) - 1 \).
\end{enumerate}
\end{corollary}

\subsection{Lagrangian submanifolds}

There is a special type of submanifolds in the hyper-Kähler manifold, called the Lagrangian submanifold. Let $X$ be a hyper-Kähler manifold of dimension $2n$.

\begin{definition}
    A Lagrangian submanifold $L\subset X$ is a submanifold of dimension $n$ such that $\sigma_X|L = 0$.
\end{definition}


Let $H_5\subset \mathbb P^6$ be a general hyperplane. Let $Y_4\subset H_5$ be a cubic $4$-fold. According to Beauville-Donagi~\cite{BD}, the moduli space of lines in $Y_4$ is a hyper-Kähler $4$-fold. Let $X = F_1(Y_4)$ be the Beauville-Donagi hyper-Kähler $4$-fold. Let $Y$ be general cubic $5$-fold containing $Y_4$ and let $\Sigma = F_2(Y)$ be the moduli space of planes in $Y$. Consider the map
\[
i: \Sigma \to X
\]
that defines as follows. Let $x\in Sigma$ be point corresponding to a plane $P_x\subset Y$. Let $L = P_x\cap H_5$ be the intersection line. Then $L$ is a line inside $Y\cap H_5 = Y_4$, hence corresponds to a point $l\in X$. Then we define $l = i(x)$.

\begin{theorem}[Iliev-Manivel~\cite{IlievManivel}]
    The map $i: \Sigma \to X$ is an embedding and sends $\Sigma$ onto a Lagrangian submanifold of $X$.
\end{theorem}

In the next section, let us study more closely the moduli $\Sigma$.

\subsection{Moduli space of planes in a cubic $5$-fold}
Let $Y\subset \mathbb P^6$ be a \emph{general} cubic $5$-fold. The moduli space $\Sigma$ of planes in $Y$ can be viewed as a subvariety of the Grassmannian $\mathrm{Gr}(3, 7)$ in the following way: Let $\mathcal U$ be the tautological subbundle of $\mathrm{Gr}(3, 7)$. The defining polynomial $f$ of $Y$ induces a global section $\sigma_f$ of the vector bundle $\mathrm{Sym}^3\mathcal U^\vee$. A plane $P\subset \mathbb P^6$ is contained in $Y$ if and only if the section $\sigma_f$ vanishes at the point $x_P\in \mathrm{Gr}(3, 7)$ representing the plane $P$. Hence, the moduli space $\Sigma$ is exactly the zero locus of the section $\sigma_f$, and thus it is a projective variety. Therefore, in order to understand the geometry of $\Sigma$, it is essential to study the vector bundle $\mathrm{Sym}^3\mathcal U^\vee$.

\begin{proposition}\label{PropSym3UdualGloballyGenerated}
     The vector bundle $\mathrm{Sym}^3\mathcal U^\vee$ is {globally generated} of rank $10$ and every section has a zero.
\end{proposition}
Recall that a vector bundle $E$ on a variety $X$ is called globally generated, if for any $x\in X$, the evaluation map 
\[
H^0(X, E) \to E|_x
\]
is surjective. Now we proceed to the proof.
\begin{proof} As the rank of $\mathcal U^\vee$ is $3$, the rank of $\mathrm{Sym}^3\mathcal U^\vee$ is $\binom{3+3-1}{3}=10$.
Now consider a point \( x \in \mathrm{Gr}(3, 7) \), corresponding to a 3-dimensional subspace \( W_x \subset \mathbb{C}^7 \). The fiber \( \mathrm{Sym}^3 \mathcal{U}^\vee |_x \) is the space of homogeneous polynomials of degree 3 on \( W_x \), i.e.,
\[
\mathrm{Sym}^3 \mathcal{U}^\vee |_x \cong \{ \text{homogeneous polynomials of degree 3 on } W_x \}.
\]
On the other hand, the global sections of \( \mathrm{Sym}^3 \mathcal{U}^\vee \), denoted by \( H^0(\mathrm{Gr}(3,7), \mathrm{Sym}^3 \mathcal{U}^\vee) \), are isomorphic to the space of homogeneous polynomials of degree 3 on \( \mathbb{C}^7 \):
\[
H^0(\mathrm{Gr}(3,7), \mathrm{Sym}^3 \mathcal{U}^\vee) \cong \{ \text{homogeneous polynomials of degree 3 on } \mathbb{C}^7 \}.
\]
The evaluation map
\[
H^0(\mathrm{Gr}(3,7), \mathrm{Sym}^3 \mathcal{U}^\vee) \to \mathrm{Sym}^3 \mathcal{U}^\vee |_x
\]
is simply the restriction of these global polynomials to the subspace \( W_x \). Since the restriction map is surjective, we conclude that $\mathrm{Sym}^3\mathcal U^\vee$ is  globally generated.

Let $I = \{ (x, \sigma)\in \mathrm{Gr}(3, 7)\times H^0(\mathrm{Gr}(3, 7), \mathrm{Sym}^3\mathcal U^\vee): \sigma(x) = 0\}$ be the incidental variety and let $q: I\to H^0(\mathrm{Gr}(3, 7), \mathrm{Sym}^3\mathcal U^\vee)$ be the second projection map. It is shown in 
\cite[Proposition 2.1]{Borcea} that the map $q$ is surjective. The latter means exactly that any global section of $\mathrm{Sym}^3\mathcal U^\vee$ has a zero.

This concludes that \( \mathrm{Sym}^3 \mathcal{U}^\vee \) is globally generated and that every section has a zero.

\end{proof}

To proceed, we will use the following classical result in algebraic geometry. We write down a proof for completeness.
\begin{proposition}\label{PropGloballyGeneratedVBHasGoodProp}
    Let $X$ be a projective variety of dimension $n$. Let $E$ be a globally generated vector bundle of rank $r\geq n$ on $X$ such that any global section on it has a zero. Let $s\in H^0(X,E)$ be a general global section of $E$. Then the zero locus of $s$ is a \emph{smooth} subvariety of dimension $n - r$.
\end{proposition}
\begin{proof}
Let 
\[ I = \{(x, s) \in X \times \mathbb{P}H^0(X, E) : s(x) = 0\}. \]

Let \( p: I \to X \) and \( q: I \to \mathbb{P}H^0(X, E) \) be the two projections. We claim that the fibers of \( p \) are projective linear subspaces of \(\mathbb{P}H^0(X, E)\) of codimension \(r\). In fact, since \( E \) is globally generated, for any \( x \in X \), the evaluation map
\[ ev_x: H^0(X, E) \to E_x \simeq \mathbb{C}^r \]
is surjective. This implies that 
\[
\dim I = n + \dim \mathbb PH^0(X, E) - r.
\]
by a dimension theorem~\cite[Ex.II.3.22(c)]{Hartshorne}.
On the other hand, the fact that every global section has a zero implies that the projection map $q$ is surjective. Hence, by the algebraic Sard’s theorem~\cite[Corollary III.10.7]{Hartshorne}, the general fiber of $q$ is smooth of dimension $\dim I - \dim \mathbb PH^0(X, E) = n - r$, as desired.
\end{proof}
Since $\mathrm{Gr}(3, 7)$ is a smooth projective variety of dimension $12$, and $\mathrm{Sym}^3 \mathcal{U}^\vee$ is a globally generated vector bundle of rank $10$ such that every global section has a zero, by Proposition~\ref{PropSym3UdualGloballyGenerated} and Proposition~\ref{PropGloballyGeneratedVBHasGoodProp}, we have the following result
\begin{corollary}
 The moduli space $\Sigma$ of planes in a general cubic $5$-fold $Y$ is a smooth projective surface.
\end{corollary}

Finally, we want to study ampleness of the canonical bundle of the surface $\Sigma$. Recall the following definition of the ampleness in general. A line bundle $L$ on a projective space $X$ induces a canonical map
\[
\begin{array}{cccc}
  \phi_L: & X & \dashrightarrow & \mathbb PH^0(X, L)^*\\
   & x & \mapsto & (\sigma\mapsto \sigma(x))
\end{array}.
\] The line bundle $L$ is called very ample, if the induced map $\phi_L$ is a closed embedding, and $L$ is called ample, if some tensor power of $L$ is very ample. For the surface $\Sigma$ in the article, we have
\begin{proposition} 
    The canonical bundle of the surface $\Sigma$ is ample.
\end{proposition}

\begin{proof} As $\Sigma$ is a smooth subvariety of $\mathrm{Gr}(3, 7)$, by the adjunction formula, we have
    \[
    K_{\Sigma} = \left. \left(K_{\mathrm{Gr}} \otimes \det(\mathrm{Sym}^3 \mathcal{U}^\vee)\right) \right|_{\Sigma},
    \]
    where $K_{\mathrm{Gr}}$ is the canonical bundle of $\mathrm{Gr}(3, 7)$, and $\det(\mathrm{Sym}^3 \mathcal{U}^\vee)$ is the determinant line bundle associated with the vector bundle $\mathrm{Sym}^3 \mathcal{U}^\vee$.
The canonical bundle of the Grassmannian \( \text{Gr}(3,7) \) is given by:
\[
K_{\text{Gr}} = \varphi^* \mathcal{O}_{\mathbb{P}(\wedge^3 \mathbb{C}^7)}(-7)
\]
where $\phi: \mathrm{Gr}(3, 7)\to \mathbb P(\bigwedge^3\mathbb C^7)$ is the Plücker embedding. Since $\det\mathcal U^\vee = \phi^*\mathcal O_\mathbb P(1)$, a formal virtual root calculation shows that $\det\mathrm{Sym}^3\mathcal U^\vee = \mathcal O_\mathbb P(10)$. Hence, $K_\Sigma = (\phi^*\mathcal O_\mathbb P(3))|_\Sigma$, which is the restriction of an ample line bundle, and is hence ample.
\end{proof}


\subsection{Lagrangian families}
Following the definition in~\cite{Bai}, a Lagrangian family in a hyper-Kähler manifold $X$ is roughly a family of Lagrangian submanifolds in $X$.

\begin{eg}
    Let $H_5\subset \mathbb P^6$ be a hyperplane and $Y_4\subset H_5$ be general cubic $4$-fold. Let $B$ be the moduli space of all cubic $5$-folds $Y\subset \mathbb P^6$ such that $Y\cap H_5 = Y_4$. By Iliev-Manivel~\cite{IlievManivel}, for a general element in $B$ which corresponds to a cubic $5$-fold $Y$, the corresponding moduli space of planes in $Y$ is a Lagrangian submanifold in $X$. Therefore, we get a family of Lagrangian submanifolds in $X$ parametrized by an open subset of $B$. 
\end{eg}

From now on, let us denote $B^0\subset B$ as the moduli space of cubic $5$-folds containing $Y_4$ whose moduli of planes is a Lagrangian submanifold of $X$. Let $\mathcal S\to B^0$ be the corresponding Lagrangian family. In order to prove Theorem~\ref{ThmVHSisMaximal}, we need to utilize the following result in~\cite[Proposition 2]{Bai}

\begin{theorem}[Bai~\cite{Bai}]
    Let $p: \mathcal L\to B$ be a family of Lagrangian submanifolds in a hyper-Kähler manifold $X$ of dimension $2n$. Suppose $X$ is very general, $b_2(X)_{tr} \geq 5$ and $h^{1, 0}(L) < 2^{ \lfloor \frac{b_2(X)_{tr} - 3}{2}\rfloor} $ for a fiber $L$ in the family $\mathcal L$. Then the variation of degree $1$ Hodge structures of the family $\mathcal L$ is maximal.
\end{theorem}

The result we are using here is a special case of the Proposition in~\cite{Bai}. Notice that if the fiber of the Lagrangian family is smooth, then the condition $\clubsuit$ is automatically satisfied, as is explained in the paragraph under~\cite[Proposition 1]{Bai}.


\section{Relations of Hodge numbers and Chern classes}

\subsection{Chern classes}\label{SubsectionChernClasses}

\begin{theorem}[Grothendieck~\cite{Grothendieck}]
There exist Chern classes $c_i(E)\in H^{2i}(X,\mathbb Z)$ for any vector bundle $E$ of rank $r$ of any smooth algebraic variety $X$, which are uniquely determined by the following properties

(i) (Vanishing) If $k > r$, then $c_k(E) = 0$.

(ii) (Whitney Sum) For a short exact sequence
\[ 0 \to E' \to E \to E'' \to 0, \]
the total Chern classes satisfy the property
\[ c(E) = c(E')c(E''). \]
In particular, $c_1(E) = c_1(E') + c_1(E'')$.


(iii) (Projective Bundle Formula) For a vector bundle $E$ over $X$, the total Chern class of the projective bundle $\mathbb P(E)$ is given by
\[ c(\mathbb P(E)) = \pi^*c(E) \cdot \sum_{i=0}^{r-1} \xi^i, \]
where $\pi: \mathbb P(E) \to X$ is the projection, $\xi = c_1(\mathcal{O}_{\mathbb P(E)}(1))$, and $r$ is the rank of $E$.

(iv) (Pull-back) For a morphism $f: Y \to X$ and a vector bundle $E$ over $X$, the Chern classes satisfy
\[ c(f^*E) = f^*c(E). \]

(v) (Normalization) If $L$ is a line bundle on $X$, then its first Chern class is given by
\[ c_1(L) = [D], \]
where $D$ is the divisor of a general rational section of $L$.
\end{theorem}

The Chern character of a vector bundle \( E \) is another invariant of $E$ that can be defined by combinations of Chern classes of $E$. The Chern character has the following expansion (dropping \( E \) from the notation for simplicity):
\[
\mathrm{ch}(E) = \mathrm{rk} + c_1 + \frac{1}{2}(c_1^2 - 2c_2) + \frac{1}{6}(c_1^3 - 3c_1c_2 + 3c_3) + \cdots
\]
We refer the readers to~\cite[A.4]{Hartshorne} for strict definition of the Chern characters and the following useful property.

\begin{proposition}\label{PropChernChar}
    Let $E, E'$ be vector bundles on $X$. Then we have
    \[
    ch(E\otimes E') = ch(E).ch(E').
    \]
\end{proposition}

\subsection{Euler characteristic of $\Sigma$}
The following version of Hopf-Poincaré theorem, which is called the Hopf index theorem as in~\cite[Theorem 11.25]{BottTu}, relates the Euler characteristic of a complex manifold to its top degree Chern class.

\begin{theorem}[Hopf-Poincaré theorem]
    Let $X$ be a compact complex manifold of dimension $n$. Let $e(X) := \sum_{i = 0}^{2n} (-1)^i b_i(X)$ be the Euler characteristic of $X$. Let $c_n(X)\in H^{2n}(X,\mathbb Z)\cong \mathbb Z$ be the top degree Chern class of $X$ viewed as an integer. Then 
    \[
    e(X) = c_n(X).
    \]
\end{theorem}
By the Hopf-Poincaré theorem, the Euler class of the surface $\Sigma$ is $c_2(\Sigma)$ viewed as an integer via the identification $\int_\Sigma: H^4(\Sigma, \mathbb Z)\cong \mathbb Z$.

\subsection{Euler characteristic of $\mathcal{O}_\Sigma$}
The Euler characteristic of the structure sheaf of a surface is also determined by its Chern classes, according to the Noether's formula that we discuss below.

Let $S$ be an arbitary smooth projective surface, we have
\begin{theorem}[Noether's formula]\label{ThmNoetherFormula}
    Let $S$ be a smooth projective surface, then we have 
    \[
    \chi(\mathcal O_S) = \frac{1}{12}(c_1(S)^2 + c_2(S)).
    \]
\end{theorem}

Modern algebraic geometry views Noether's formula as a special case of the Grothendieck-Riemann-Roch formula whose proof is beyond the scope of this project. More elementary proof of the Noether's formula exists in the litterature (see, for example, \cite[Section 4.6]{GriffithsHarris}). Here we follow the first step in~\cite[Section 4.6]{GriffithsHarris} that proves Noether's formula for surfaces in $\mathbb P^3$. We refer the readers to~\cite[Section 4.6]{GriffithsHarris} for the proof of Noether's formula for a general smooth projective surface (not necessarily in $\mathbb P^3$). To start with, we calculate the Chern numbers of $S\subset \mathbb P^3$.

\begin{lemma}\label{LmmChernClassesOfSurfaceInP3}
For a smooth surface \( S \) of degree \( d \) in \(\mathbb{P}^3\),
the first Chern class and the second Chern class of the tangent bundle $T_S$ over $S$ are:
\[ c_1(S) = (4-d)H. \]
\[ c_2(S) = d^3 - 4d^2 + 6d. \]
Moreover, we have $c_1(S)^2 = d^3 - 8d^2 + 16d$.
\end{lemma}
\begin{proof} 
Consider the short exact sequence
\[ 
0 \to T_S \to T_{\mathbb{P}^3}|_S \to N_{S|\mathbb{P}^3} \to 0. 
\]
Here, \(T_S\) is the tangent bundle of \( S \), \( T_{\mathbb{P}^3}|_S \) is the restriction of the tangent bundle $T_{\mathbb{P}^3}$ of \(\mathbb{P}^3\) to \( S \), and \( N_{S|\mathbb{P}^3} \) is the normal bundle of \( S \) in \(\mathbb{P}^3\). From the above sequence, we have 
$$c(T_{\mathbb{P}^3}|_S) = c(S) \cdot c(N_{S|\mathbb{P}^3}).$$ 
We know that the total Chern class of \( T_{\mathbb{P}^3} \) is $c(T_{\mathbb{P}^3}) = (1 + h)^4$,
where \( h \) is the hyperplane class in \(\mathbb{P}^3\). Restricting to \( S \), we get 
$$c(T_{\mathbb{P}^3}|_S) = (1 + h|_S)^4=(1+H)^4$$ 
where $H:= h|_S $.  

Now, the normal bundle \( N_{S/\mathbb{P}^3} \) is \(\mathcal{O}(d) \), so its total Chern class is
\[ c(N_{S|\mathbb{P}^3}) = 1 + dH. \]
Hence we have 
\[ (1 + H)^4 = c(S) \cdot (1 + dH). \]
Expanding and comparing terms, we get
\[
c_1(S) = (4 - d)H, 
\]
\[
c_2(S)  = (6 - 4d + d^2)H^2. 
\]

Hence  \[ 
c_1(S)^2 = ((4 - d)H)^2 = (4 - d)^2H^2 = (16 - 8d + d^2)H^2.
\]
Since \( H^2 \) represents the self-intersection of two hyperplane sections on a smooth surface of degree \( d \) in \(\mathbb{P}^3\), we have \( H^2 = d \). Therefore we have
$$ 
c_1(S)^2 =d^3 - 8d^2 + 16d.
$$
and
$$c_2(S) = d^3 - 4d^2 + 6d. $$
\end{proof}
Now we can prove the Noether's formula in the special case where $S$ is a surface in $\mathbb P^3$.
\begin{proof}[Proof of Noether's formula when $S\subset \mathbb P^3$]
Firstly, we consider the short exact sequence
$$0 \to \mathcal{O}_{\mathbb{P}^3}(-d) \to \mathcal{O}_{\mathbb{P}^3}\to \mathcal{O}_S \to 0.$$
According to the property of Euler characteristic, we have
$$\chi (\mathcal{O}_S) = \chi(\mathcal{O}_{\mathbb{P}^3}) - \chi(\mathcal{O}_{\mathbb{P}^3}(-d)).$$ 
According to \cite[Theorem II.5.1]{Hartshorne}, the  Euler characteristic for $\mathcal{O}_{\mathbb{P}^n}$(d) is given by
$$\chi (\mathcal{O}_{\mathbb{P}^n}(d) )= \binom{d+n}{n}.$$
In particular, we have:
$$\chi(O_{\mathbb{P}^3}) = 1,$$
$$\chi(O_{\mathbb{P}^3}(-d)) = \frac{(-d+3)(-d+2)(-d+1)}{6}.$$
Hence we have
$$\chi (\mathcal{O}_S) = 1 -  \frac{(-d+3)(-d+2)(-d+1)}{6}  = \frac{d^3 - 6d^2 + 11d}{6}.$$
However, by Lemma~\ref{LmmChernClassesOfSurfaceInP3}, we have
$$\frac{1}{12}(c_1(S)^2+c_2(S)) = \frac{2d^3 - 12d^2 + 22d}{12} = \frac{d^3 - 6d^2 + 11d}{6}. $$
Therefore, $$\chi (\mathcal{O}_S) = \frac{1}{12}(c_1(S)^2+c_2(S)) .$$
\end{proof}

Applying Theorem~\ref{ThmNoetherFormula} directly to the surface $\Sigma$, the Euler characteristic $\chi(\mathcal O_\Sigma)$ of the surface $\Sigma$ is $\frac{1}{12}(c_1(\Sigma)^2 + c_2(\Sigma))$.

\subsection{The Abel-Jacobi map of the surface}

\begin{definition}
The \textit{intermediate Jacobian} \( J^{2k - 1}(X) \) for a smooth projective variety \( X \) is defined as
\[
J^{2k - 1}(X) = H^{2k-1}(X, \mathbb{C}) / (F^k H^{2k-1}(X) + H^{2k-1}(X, \mathbb{Z}))
\]
where \( F^k \) denotes the \( k \)-th level of the Hodge filtration.
\end{definition}

\begin{definition}
The \textit{Albanese variety} \( \text{Alb}(X) \) of a smooth projective variety \( X\) is defined as:
\[
\text{Alb}(X) = H^0(X, \Omega^1_X)^* / H_1(X, \mathbb{Z}).
\]
\end{definition}

\begin{definition} Let $X$ be a smooth projective variety and $x_0$ is a point on $X$, then the \textit{Albanese map} \( \mathrm{alb}: X\to \text{Alb}(X) \) is a morphism defined by integrating holomorphic 1-forms on the path connecting an arbitrary point $x\in X$ and the fixed point $x_0$.
\end{definition}
 The Albanese map is universal among morphisms from \( X\) to abelian varieties, namely, for any abelian variety $A$ and any morphism $\phi: X \to A$ that sends $x_0$ to $0\in A$, there is a unique morphism $\mathrm{Alb}(X) \to A$ such that $\phi$ factorizes through this morphism. Let \( J^5(Y) \) be the intermediate Jacobian of the cubic 5-fold \( Y \), and \( \text{Alb}(\Sigma) \) the Albanese variety of the surface \( \Sigma \). Consider the Abel-Jacobi map 
\[
a: \text{Alb}(\Sigma) \to J^5Y,
\]
defined by the universal property of the Albanese map of $\Sigma$.

The following theorem of Collino~\cite{Collino} is crucial in our calculations of Hodge numbers of the surface $\Sigma$.
\begin{theorem}[Collino~\cite{Collino}]\label{ThmCollino}
The Abel-Jacobi map \( a: \text{Alb}(\Sigma) \to J^5(Y) \) is an isomorphism.
\end{theorem}


\subsection{Hodge numbers as Chern classes}

\begin{theorem}\label{3.18}
    The Betti numbers of the surface $\Sigma$ are as follows.

    (i) $b_0(\Sigma) = b_4(\Sigma) = 1$.
    
    (ii) $b_1(\Sigma) = b_3(\Sigma) = 42$.
    
    (iii) $b_2(\Sigma) = c_2(\Sigma) + 82$.
\end{theorem}
\begin{proof}
    (i) is obvious since the surface $\Sigma$ is connected~\cite[Théorème 2.1 (c)]{DM}. 
 For (ii), we first prove that $b_3(\Sigma) = b_5(Y)$ . 
    By Theorem~\ref{ThmCollino}, the Abel-Jacobi map \( a: \text{Alb}(\Sigma) \to J^5(Y) \) is an isomorphism. This implies that the dimensions of the Albanese variety \( \text{Alb}(\Sigma) \) and the intermediate Jacobian \( J^5(Y) \) are equal. We can compute the dimension of the Albanese variety \( \text{Alb}(\Sigma) \) for the surface \( \Sigma \) as follows
\[
\dim\text{Alb}(\Sigma) = \dim H^{2,1}(\Sigma, \mathbb{C}) = \frac{1}{2}\dim H^3(\Sigma, \mathbb{C}).
\]
On the other hand, the intermediate Jacobian \(J^5(Y) \) of the cubic 5-fold \( Y \) is defined as
\[
J^5(Y) = H^{3,2}(Y, \mathbb{C}) \oplus H^{4,1}(Y, \mathbb{C}) \oplus H^{5,0}(Y, \mathbb{C}),
\]
so the dimension of \( J^5(Y) \) is
\[
\dim(J^5(Y)) = \frac{1}{2} \dim H^5(Y, \mathbb{C}).
\]
Since \( \text{Alb}(\Sigma) \cong J^5(Y) \), we have
\[
\dim H^3(\Sigma, \mathbb{C}) = \dim H^5(Y, \mathbb{C}).
\]
Next, by~\cite[Corollary 1.12]{HuybrechtsCubic}, we find $b_5(Y) = 42$. Hence, $b_3(\Sigma0 = 42$.
    Finally, according to the Poincaré duality, we have $b_1(\Sigma) = b_3(\Sigma) = 42$. As for (iii), the Hopf-Poincaré theorem shows that $\sum_{i = 0}^4 (-1)^i b_i(\Sigma) = c_2(\Sigma)$. By (i) and (ii), we find $b_2(\Sigma) = c_2(\Sigma) + 82$.
\end{proof}

\begin{theorem}\label{ThmHodgeAndChern}
    The Hodge numbers of the surface $\Sigma$ is as follows.\\
    (i) $h^{0, 0}(\Sigma) = h^{2, 2}(\Sigma) = 1$.\\
    (ii) $h^{1, 0}(\Sigma) = h^{0, 1}(\Sigma) = h^{2, 1}(\Sigma) = h^{1, 2}(\Sigma) = 21$.\\
    (iii) $h^{2, 0}(\Sigma) = h^{0, 2}(\Sigma) = \frac{1}{12}(c_1(\Sigma)^2 + c_2(\Sigma)) + 20$.\\
    (iv) $h^{1, 1}(\Sigma) = \frac{1}{6}(5c_2(\Sigma) - c_1(\Sigma)^2) + 42$.
\end{theorem}
\begin{proof}
    (i) is obvious since $\Sigma$ is a connected surface. For (ii), Theorem~\ref{ThmCollino} and the table at the end of \cite[Section 1.4]{HuybrechtsCubic} show that $h^{3,2}(Y) = h^{2,1}(\Sigma) = 21$. The other three Hodge numbers are also $21$ by the symmetries mentioned below Defintion~\ref{DefHodgeNumbers}. As for (iii), by Hodge symmetry, $h^{2, 0}(\Sigma) = h^{0, 2}(\Sigma) $. Hence, it suffices to compute $h^{0, 2}(\Sigma)$. By Noether's formula, we have $h^{0, 0}(\Sigma) - h^{0, 1}(\Sigma) + h^{0, 2}(\Sigma) = \chi(\mathcal O_\Sigma) = \frac{1}{12}(c_1(\Sigma)^2 + c_2(\Sigma))$. Hence, by (i) and (ii), we get $h^{0, 2}(\Sigma) =  \frac{1}{12}(c_1(\Sigma)^2 + c_2(\Sigma)) + 20$. Finally, we notice that the Betti number \( b_2 \) for a surface \(\Sigma\) can be decomposed in terms of Hodge numbers as: $b_2(\Sigma) = h^{2,0}(\Sigma) + h^{0,2}(\Sigma) + h^{1,1}(\Sigma)$. As we know the value of $h^{2,0}(\Sigma)$ and $h^{0,2}(\Sigma)$ in (iii), we find that the value of $h^{1,1}(\Sigma)$ is $\frac{1}{6}(5c_2(\Sigma) - c_1(\Sigma)^2) + 42$.
\end{proof}

In this section, we have expressed all the Betti numbers and Hodge numbers as expressions of $c_1(\Sigma)^2$ and $c_2(\Sigma)$. Therefore, it remains to calculate $c_1(\Sigma)^2$ and $c_2(\Sigma)$ as integers. This is done in the next section.

\section{Calculations of $c_1(\Sigma)^2$ and $c_2(\Sigma)$}\label{SectionChernCalculations}

\subsection{Schubert Cycles and Schubert Calculus}\label{SubsectionSchubert}
This section recalls the basic theory of the Schubert calculus which is the main calculation tool for the Chern classes. The presentation here follows closely~\cite{Fulton}.

Let \( \text{Gr}(k, n) \) denote the Grassmannian parametrizing \( k \)-dimensional subspaces of an \( n \)-dimensional vector space \( \mathbb{C}^n \). The Grassmannian has a complex manifold structure. The complex dimension of \( \text{Gr}(k, n) \) is \( k(n-k) \).
\begin{definition} A partition $\lambda=(\lambda_1,...,\lambda_k)$ is a weakly decreasing sequence of non-negative integers. We denote the size of $\lambda$ by $|\lambda|=\sum_{i=1}^k\lambda_i$. 

Two partitions $\lambda$ and $\mu$ are said to be complentary if $\lambda_i+\mu_{k+1-i}=n-k$ for all $i$.
\end{definition}
\begin{definition} A \textit{complete flag} \( F \) in \( \mathbb{C}^n \) is a sequence of nested subspaces:
\[
F_\bullet: \quad \{0\} = F_0 \subset F_1 \subset F_2 \subset \cdots \subset F_n = \mathbb{C}^n,
\]
where \( F_i \) is an \( i \)-dimensional subspace of \( \mathbb{C}^n \) for each \( i \). 
\end{definition}
\begin{definition} For a complete flag $F$ and a partition $\lambda$, the associated Schubert cell is defined as 
$$\Omega_\lambda^\circ(F) = \{V \in \text{Gr}(k, n) : \dim(V \cap F_r)=i,\ \mathrm{for}\ n-k+i-\lambda_i\leq r\leq n-k+i-\lambda_{i+1}, \forall i \}.$$

The associated Schubert variety with respect to $F$ and $\lambda$ is defined as
$$\Omega_\lambda(F) = \{V \in \text{Gr}(k, n) : \dim(V \cap F_{n-k+i-\lambda_i})\geq i,\ \forall i \}.$$
\end{definition}
\begin{definition} The fundamental class $[\Omega_\lambda(F) ]$ of a Schubert variety $\Omega_\lambda(F) $ is called a Schubert class and is denoted by $\sigma_\lambda$. If $\lambda=(\lambda_1, 0,...,0)$, then we say $\sigma_\lambda$ is a special Schubert class.
\end{definition}
The Schubert cells give a cell complex structure on the Grassmannian $\text{Gr}(k, n)$. 
\begin{eg} The Grassmannian $\mathrm{Gr}(2, 4)$ has a cellular decomposition
$$\mathrm{Gr}(2, 4)=\Omega_{(2, 2)}^\circ(F)\cup\Omega_{(2, 1)}^\circ(F)\cup\Omega_{(2, 0)}^\circ(F)\cup\Omega_{(1, 1)}^\circ(F)\cup\Omega_{(1, 0)}^\circ(F)\cup\Omega_\emptyset^\circ(F)$$
where the zero skeleton is the point $\Omega^\circ_{(2, 2)}$.
\end{eg}
By the general theory of cell complex, we have the following proposition.
\begin{proposition} As a $\mathbb{Z}$-module, the cohomology ring $H^*(\text{Gr}(k, n))$ is generated by Schubert classes
$$H^*(\text{Gr}(k, n))=\oplus_\lambda\mathbb{Z}\sigma_\lambda.$$
\end{proposition}

The Schubert cell \( \Omega_\lambda^\circ(F_\bullet) \) is isomorphic to an affine space of dimension \( \sum_{i=1}^k (n-k+a_i-i) \).

Now we have the following results for the ring $H^*(\text{Gr}(k, n))$.
\begin{proposition}[Poincar\'e duality] Let $\lambda$ and $\mu$ be two partitions contained in an $k\times (n-k)$ rectangle such that $|\lambda|+|\mu|=k(n-k)$, then we have $$\sigma_\lambda\cup\sigma_\mu=\sigma_{\mathrm{pt}}$$ if $\lambda$ and $\mu$ are complementary and 
$$\sigma_\lambda\cup\sigma_\mu=0$$ 
otherwise.
\end{proposition}
\begin{proposition}[Pieri rule] For two partitions $\mu=(\mu_1,0,...,0)$ and $\lambda=(\lambda_1,..,\lambda_k)$, we have
$$\sigma_\mu \cdot \sigma_\lambda = \sum_r \sigma_r,
$$
where the sum is over all partitions $r$ such that $|r|=|\lambda|+|\mu|$ and $\lambda_i\leq r_i\leq\lambda_{i-1}$.
\end{proposition}
\begin{proposition}[Giambelli Formula] Any Schubert class $\sigma_\lambda\in H^*(\text{Gr}(k, n))$ can be expressed as $\sigma_\lambda=\det(\sigma_{\lambda_i+j-i})_{1\leq i, j\leq k}$ where we set $\sigma_p=1$ if $p=0$ and $\sigma_p=0$ if $p<0$ or $p>n-k$.
\end{proposition}
Using the Pieri rule and the Giambelli formula, we can compute any cup product $\sigma_\lambda\cup\sigma_\mu$. First, by using the Giambelli fomula, we express $\sigma_\lambda$ in terms of special Schubert classes, then we can compute $\sigma_\lambda\cup\sigma_\mu$ by the Pieri rule.


\subsection{Chern classes of $\Sigma$ as Chern classes in the Grassmannian}
In this section, we show that the Chern classes of $\Sigma$ can be viewed as intersections of the Chern classes of the dual of the tautological subbundle $U^\vee$ and those of the tautological quotient bundle $Q$ on $\mathrm{Gr}(3, 7)$. The explicit expression is shown in Proposition~\ref{PropChernClassesAsIntersectionOfChernClasses}.

\begin{proposition}\label{PropChernClassesAsRestriction}

the Chern classes of $\Sigma$ can be expressed as
\[
c_1(\Sigma) = \left( c_1(U^\vee \otimes Q) - c_1(\mathrm{Sym}^3 U^\vee) \right)|_\Sigma,
\]
and
\[
c_2(\Sigma) = \left( c_1(\mathrm{Sym}^3 U^\vee)^2 - c_1(U^\vee \otimes Q) \cdot c_1(\mathrm{Sym}^3 U^\vee) 
+ c_2(U^\vee \otimes Q) - c_2(\mathrm{Sym}^3 U^\vee) \right)|_\Sigma.
\]
\end{proposition}

\begin{proof}
The normal sequence of $\Sigma\subset \mathrm{Gr}(3, 7)$ gives 
\[
0 \to T_\Sigma \to T_{\mathrm{Gr}(3, 7)|\Sigma} \to N_{\Sigma/\mathrm{Gr}(3, 7)}\to 0.
\]
Since $S\subset \mathrm{Gr}(3, 7)$ is given by the zero locus of a general section of the vector bundle $\mathrm{Sym}^3U^\vee$, we have $N_{\Sigma/\mathrm{Gr|(3, 7)}}\cong \mathrm{Sym}^3U^\vee|\Sigma$. Furthermore, it is well known~\cite[Theorem 3.5]{3264} that the tangent bundle of the Grassmannian $\mathrm{Gr}(k, n)$ is canonically isomorphic to $U^\vee\otimes Q$. Therefore, the normal sequence of $\Sigma\subset \mathrm{Gr}(3, 7)$ becomes  
\[
0 \to T_\Sigma \to (U^\vee \otimes Q)|_\Sigma \to (\mathrm{Sym}^3 U^\vee)|_\Sigma \to 0.
\]
By applying the Whitney sum formula for Chern classes, we have:
\[
c((U^\vee \otimes Q)|_\Sigma) = c(T_\Sigma) \cdot c((\mathrm{Sym}^3 U^\vee)|_\Sigma).
\]
Expanding this product using the definition of the total Chern class
\[
c(E) = 1 + c_1(E) + c_2(E) + \dots,
\]
we obtain the following relations for the first and second Chern classes.
For the first Chern class, we get
\[
c_1((U^\vee \otimes Q)|_\Sigma) = c_1(T_\Sigma) + c_1((\mathrm{Sym}^3 U^\vee)|_\Sigma).
\]
Thus, solving for $c_1(T_\Sigma)$, we have
\[
c_1(T_\Sigma) = c_1((U^\vee \otimes Q)|_\Sigma) - c_1((\mathrm{Sym}^3 U^\vee)|_\Sigma).
\]
At this point, we use the fact that \( T_\Sigma \) is the tangent bundle of \( \Sigma \), which implies that its Chern classes are exactly those of \( \Sigma \). In other words
\[
c(T_\Sigma) = c(\Sigma).
\]
Therefore, we conclude
\[
c_1(\Sigma) = c_1((U^\vee \otimes Q)|_\Sigma) - c_1((\mathrm{Sym}^3 U^\vee)|_\Sigma).
\]
For the second Chern class, the Whitney sum formula gives
\[
c_2((U^\vee \otimes Q)|_\Sigma) = c_2(T_\Sigma) + c_1(T_\Sigma) \cdot c_1((\mathrm{Sym}^3 U^\vee)|_\Sigma) + c_2((\mathrm{Sym}^3 U^\vee)|_\Sigma).
\]
Substituting the expression for $c_1(T_\Sigma)$ from the previous step, we find
\[
c_2(T_\Sigma) = c_2((U^\vee \otimes Q)|_\Sigma) - (c_1((U^\vee \otimes Q)|_\Sigma) - c_1((\mathrm{Sym}^3 U^\vee)|_\Sigma)) \cdot c_1((\mathrm{Sym}^3 U^\vee)|_\Sigma) + c_2((\mathrm{Sym}^3 U^\vee)|_\Sigma).
\]
Since \( c_2(T_\Sigma) = c_2(\Sigma) \), we have
\[
c_2(\Sigma) = \left( c_1(\mathrm{Sym}^3 U^\vee)^2 - c_1(U^\vee \otimes Q) \cdot c_1(\mathrm{Sym}^3 U^\vee)
+ c_2(U^\vee \otimes Q) - c_2(\mathrm{Sym}^3 U^\vee) \right)|_\Sigma.
\]
This completes the proof.
\end{proof}

The following theorem is well-known and its proof can be found in \cite[(5.14) p. 51]{BottTu}.

\begin{theorem} \cite{BottTu} 
Let $X$ be a differential manifold of real dimension $N$. Let $Y \subset X$ be a closed submanifold of real codimension $c$. Let $\omega \in H^{N-c}(X, \mathbb{Z})$. Then
\[
\int_Y \omega|_Y = \int_X \omega \cdot [Y].
\]
\end{theorem}

\begin{theorem}
Let \( Z \subset X \) be the zero locus of a vector bundle \( E \) of rank \( r \). Then:
\[
[Z] = c_r(E) \in H^{2r}(X, \mathbb{Z}),
\]
where \( [Z] \) is the cohomology class represented by the zero locus \( Z \).
\end{theorem}

\begin{proposition}\label{PropChernClassesAsIntersectionOfChernClasses}
Viewing $c_1(\Sigma)^2$ and $c_2(\Sigma)$ as integers, we have
\[
c_1(\Sigma)^2 = \left( c_1(U^\vee \otimes Q) - c_1(\mathrm{Sym}^3 U^\vee) \right)^2 \cdot c_{10}(\mathrm{Sym}^3 U^\vee),
\]
and
\[
c_2(\Sigma) = \left( c_1(\mathrm{Sym}^3 U^\vee)^2 - c_1(U^\vee \otimes Q) \cdot c_1(\mathrm{Sym}^3 U^\vee) + c_2(U^\vee \otimes Q) - c_2(\mathrm{Sym}^3 U^\vee) \right) \cdot c_{10}(\mathrm{Sym}^3 U^\vee).
\]
\end{proposition}

\begin{proof}
By the Theorem 4.11, we have:
\[
\int_\Sigma c_1(\Sigma)^2|_\Sigma = \int_{\mathrm{Gr}(3,7)} c_1(\Sigma)^2 \cdot [\Sigma].
\]
From the Theorem 4.12, since the rank of $\mathrm{Sym}^3 U^\vee$ equals to 10, we have $[\Sigma] = c_{10}(\mathrm{Sym}^3 U^\vee)$.
Thus, we can substitute into the integral:
\[
\int_\Sigma c_1(\Sigma)^2|_\Sigma = \int_{\mathrm{Gr}(3,7)} \left( c_1(U^\vee \otimes Q) - c_1(\mathrm{Sym}^3 U^\vee) \right)^2 \cdot c_{10}(\mathrm{Sym}^3 U^\vee).
\]
This directly proves that:
\[
c_1(\Sigma)^2 = \left( c_1(U^\vee \otimes Q) - c_1(\mathrm{Sym}^3 U^\vee) \right)^2 \cdot c_{10}(\mathrm{Sym}^3 U^\vee).
\]
Similarly, for \( c_2(\Sigma) \), we apply the Theorem 4.11
\[
\int_\Sigma c_2(\Sigma)|_\Sigma = \int_{\mathrm{Gr}(3,7)} c_2(\Sigma) \cdot [\Sigma].
\]
Then we have 
$$\int_\Sigma c_2(\Sigma)|_\Sigma = \int_{\mathrm{Gr}(3,7)} \left( c_1(\mathrm{Sym}^3 U^\vee)^2 - c_1(U^\vee \otimes Q) \cdot c_1(\mathrm{Sym}^3 U^\vee) + c_2(U^\vee \otimes Q) - c_2(\mathrm{Sym}^3 U^\vee) \right) \cdot c_{10}(\mathrm{Sym}^3 U^\vee).$$
This proves that
\[
c_2(\Sigma) = \left( c_1(\mathrm{Sym}^3 U^\vee)^2 - c_1(U^\vee \otimes Q) \cdot c_1(\mathrm{Sym}^3 U^\vee) + c_2(U^\vee \otimes Q) - c_2(\mathrm{Sym}^3 U^\vee) \right) \cdot c_{10}(\mathrm{Sym}^3 U^\vee).
\]
\end{proof}

\subsection{Chern classes as Schubert classes on the Grassmannian}

Since the Schubert classes $\sigma_\lambda$ form a basis for $H^*(Gr, \mathbb Z)$, we are able to represent the Chern classes of $U^\vee$ and $Q$ as Schubert classes. In fact, it is well-known~\cite[Section 5.6.2]{3264} that 
\[
c(U^\vee) = 1 + \sigma_1 + \sigma_{1, 1} + \ldots + \sigma_{1, 1, \ldots, 1},
\] 
and that
\[
c(Q) = 1 + \sigma_1 + \sigma_2 + \ldots + \sigma_{n - k}.
\]

\begin{proposition}\label{PropChernClassOfTensor}
We have $c_1(U^\vee\otimes Q) = 7\sigma_1$ and $c_2(U^\vee\otimes Q) = 24\sigma_{11} + 23\sigma_2$.
\end{proposition}

\begin{proof}
Using Proposition~\ref{PropChernChar}, we get
\[ 
\mathrm{ch}(U^\vee \otimes Q)= \left( 4 + \sigma_1 + \frac{1}{2} \left( \sigma_{11} + \sigma_{20} - 2\sigma_2 + \cdots \right) \right)
\left( 3 + \sigma_1 + \frac{1}{2} \left( \sigma_{11} + \sigma_{20} + 2\sigma_{11} + \cdots \right) \right).
\]
By collecting the degree 1 factors, we have
\[
c_1(U^\vee \otimes Q) = 4\sigma_1 + 3\sigma_1 = 7\sigma_1.
\]
By collecting the degree 2 factors, we have
\[
\frac{1}{2} c_1(U^\vee \otimes Q)^2 - 2c_2(U^\vee \otimes Q) = \frac{1}{2} (\sigma_2 - \sigma_{11}) + 2\sigma_1^2.
\]
Simplifying the equation, we get
\[
c_2(U^\vee \otimes Q) = 24\sigma_{11} + 23\sigma_2.
\]
\end{proof}

\begin{proposition}\label{PropChernClassOfSym}
We have $c_1(\mathrm{Sym}^3 U^\vee) = 10 \sigma_1$, 
$c_2(\mathrm{Sym}^3 U^\vee) = 40 \sigma_1^2 + 15 \sigma_{11}$, and 
$c_{10}(\mathrm{Sym}^3 U^\vee) = 216 \sigma_1^3 \sigma_{11}^2 \sigma_{111} + 108 \sigma_1 \sigma_{11}^3 \sigma_{111} + 108 \sigma_1^4 \sigma_{111}^2 - 486 \sigma_1^2 \sigma_{11} \sigma_{111}^2 + 729 \sigma_1 \sigma_{111}^3$.
\end{proposition}

\begin{proof}
Assume the virtual roots of \(U^\vee\) are \(l_1, l_2, l_3\). The virtual roots of \(\mathrm{Sym}^3 U^\vee\) are given by \(l_{i_1} + l_{i_2} + l_{i_3}\), where the indices satisfy \(1 \leq i_1 \leq i_2 \leq i_3 \leq 3\).
The number of terms is
\[
\binom{3+2}{3} = \binom{5}{3} = 10.
\]
Thus, the total Chern class of \(\mathrm{Sym}^3 U^\vee\) is
\[
c(\mathrm{Sym}^3 U^\vee) = \prod_{1 \leq i_1 \leq i_2 \leq i_3 \leq 3} (1 + l_{i_1} + l_{i_2} + l_{i_3}),
\]
leading to the expansion into individual Chern classes
\[
c(\mathrm{Sym}^3 U^\vee) = 1 + c_1(\mathrm{Sym}^3 U^\vee) + c_2(\mathrm{Sym}^3 U^\vee) + \cdots + c_{10}(\mathrm{Sym}^3 U^\vee).
\]

(i) The first Chern class is sum of the symmetric polynomials of degree 1
\[
c_1(\mathrm{Sym}^3 U^\vee) = \sum_{1 \leq i_1 \leq i_2 \leq i_3 \leq 3} (1 + l_{i_1} + l_{i_2} + l_{i_3}).
\]
This simplifies to
\[
c_1(\mathrm{Sym}^3 U^\vee) = 10 (l_1 + l_2 + l_3).
\]
Since \(c_1(U^\vee) = l_1 + l_2 + l_3 = \sigma_1\), we have
\[
c_1(\mathrm{Sym}^3 U^\vee) = 10 \sigma_1.
\]

(ii) The second Chern class is sum of the symmetric polynomials of degree 2 in \(l_1, l_2, l_3\), which takes the form
\[
c_2(\mathrm{Sym}^3 U^\vee) = a \cdot (l_1 + l_2 + l_3)^2 + b \cdot (l_1 l_2 + l_2 l_3 + l_3 l_1),
\]
where \(a\) and \(b\) are constants. This simplifies to
\[
c_2(\mathrm{Sym}^3 U^\vee) = a \cdot \sigma_1^2 + b \cdot \sigma_{11}.
\]
The coefficients $a$ and $b$ can be determined as follows. The coefficient $a$ is the term of $l_1^2$. Among the $10$ terms $l_{i_1} + l_{i_2} + l_{i_3}$ in the product, there are $4$ terms with no $l_1$, $3$ terms with $1$ $l_1$, $2$ terms with $2$ $l_1$’s and $1$ term with $3$ $l_1$’s. Hence, the coefficient of $l_1$ in the expansion is $\binom{3}{2}\times 1\times 1 + \binom{2}{2}\times 2\times 2  + 3\times 2\times 1\times 2 + 3\times 1\times 1\times 3 + 2\times 1\times 2\times 3 = 40$. Therefore, $a = 40$. By similar counting methods, the coefficient of the term $l_1l_2$ in the expansion is $95$. Hence, $b = 95 - 2a = 15$. Therefore, we get
\[
c_2(\mathrm{Sym}^3 U^\vee) = 40 \sigma_1^2 + 15 \sigma_{11}.
\]
(iii)
The degree $10$ part of the expansion of $\prod_{1 \leq i_1 \leq i_2 \leq i_3 \leq 3} (1 + l_{i_1} + l_{i_2} + l_{i_3})$ is simply 
\[
\prod_{1\leq i_1\leq i_2\leq i_3\leq 3}(l_{i_1} + l_{i_2} + l_{i_3}).
\]
To express $c_{10}(\mathrm{Sym}^3U^\vee)$ as Chern classes of $U^\vee$, it suffices to express the symmetric polynomial $\prod_{1\leq i_1\leq i_2\leq i_3\leq 3}(l_{i_1} + l_{i_2} + l_{i_3})$ as elementary symmetric polynomials. We do so using Mathematica with the following code
\begin{verbatim}
    eu = (3 l1) (3 l2) (3 l3) (2 l1 + l2) (2 l1 + l3) (2 l2 + l3) (2 l2 + 
     l1) (2 l3 + l1) (2 l3 + l2) (l1 + l2 + l3);
    SymmetricReduction[eu, Variables[eu], {c1, c2, c3}]
\end{verbatim}
The output is
\begin{verbatim}
    {216 c1^3 c2^2 c3 + 108 c1 c2^3 c3 + 108 c1^4 c3^2 - 486 c1^2 c2 c3^2 + 
    729 c1 c3^3, 0}
\end{verbatim}
This terminates the calculation.
\end{proof}

\subsection{Explicit Schubert calculus on $\mathrm{Gr}(3, 7)$}
\begin{lemma}\label{LmmExplicitSchubertCalculus}
We have the following results:

(i) $\sigma_1^6 \cdot \sigma_{111}^2 = 5$;

(ii) $\sigma_1^5 \cdot \sigma_{11}^2 \cdot \sigma_{111} = 11$;

(iii) $\sigma_1^3 \cdot \sigma_{111} \cdot \sigma_{11}^3 = 6$;

(iv) $\sigma_1^4 \cdot \sigma_{11} \cdot \sigma_{111}^2 = 3$;

(v) $\sigma_1^3 \cdot \sigma_{111}^3 = 1$;

(vi) $\sigma_1 \cdot \sigma_{11}^4 \cdot \sigma_{111} = 3$;

(vii) $\sigma_1^2 \cdot \sigma_{11}^2 \cdot \sigma_{111}^2 = 2$;

(viii) $\sigma_1 \cdot \sigma_{11} \cdot \sigma_{111}^3 = 1$.
\end{lemma}

\begin{proof}
The proof of this lemma is straightforward. We only need to perform some calculations based on Schubert calculus. Here is an example.

(i) We start with the following equation:
\[
\sigma_1^6 \cdot \sigma_{111} = \sigma_1^5 \cdot \sigma_1 \cdot \sigma_{111}
\]
Using Pieri's rule, we proceed step by step:

\begin{align*}
\sigma_1^5 \cdot \sigma_{211} &= \sigma_1^4 \cdot (\sigma_{311} + \sigma_{221}) \\
&= \sigma_1^3 \cdot (\sigma_{411} + 2\sigma_{321} + \sigma_{222}) \\
&= \sigma_1^2 \cdot (3\sigma_{421} + 3\sigma_{322} + 2\sigma_{331}) \\
&= \sigma_1 \cdot (5\sigma_{431} + 6\sigma_{422} + 5\sigma_{332}) \\
&= \sigma_1 \cdot (5\sigma_{441} + 16\sigma_{432} + 5\sigma_{333}).
\end{align*}
Finally, applying Poincaré duality, we calculate:
\[
\sigma_1^6 \cdot \sigma_{111} \cdot \sigma_{111} = (5\sigma_{441} + 16\sigma_{432} + 5\sigma_{333}) \cdot \sigma_{111}
= 5 \cdot \sigma_{333} \cdot \sigma_{111} = 5.
\]
This process proves the result for equation (i). The proofs for the other five equations follow similarly, and hence are omitted.
\end{proof}

\subsection{Proof of Theorem~\ref{ThmBettiNumbersIntro} and Theorem~\ref{ThmHodgeNumbersIntro}}

\begin{theorem}\label{ThmChernClassesAsNumbers}
Under the canonical identification $\int: H^4(\Sigma, \mathbb Z)\cong \mathbb Z$, we have
\[
c_1{(\Sigma)}^2 = 25515
\]
\[
c_2{(\Sigma)} = 13041
\]
\end{theorem}

\begin{proof}
By Proposition~\ref{PropChernClassesAsIntersectionOfChernClasses}, Proposition~\ref{PropChernClassOfTensor}, Proposition~\ref{PropChernClassOfSym} and Lemma~\ref{LmmExplicitSchubertCalculus}, we get
\begin{align*}
c_1{(\Sigma)}^2 &= (7\sigma_1 - 10\sigma_1)^2 \cdot \left(216 \sigma_1^3 \sigma_{11}^2 \sigma_{111} + 108 \sigma_1 \sigma_{11}^3 \sigma_{111} + 108 \sigma_1^4 \sigma_{111}^2 - 486 \sigma_1^2 \sigma_{11} \sigma_{111}^2 + 729 \sigma_1 \sigma_{111}^3\right) \\
&= 1944 \sigma_1^3 \sigma_{11}^2 \sigma_{111} + 972 \sigma_1 \sigma_{11}^3 \sigma_{111} + 972 \sigma_1^4 \sigma_{111}^2 - 4374 \sigma_1^2 \sigma_{11} \sigma_{111}^2 + 6561 \sigma_1 \sigma_{111}^3
\\&= 25515,
\end{align*}
and 
\begin{align*}
c_2(\Sigma) = & (100 \sigma_1^2 - 70 \sigma_1^2 + (23 \sigma_1^2 + \sigma_{11}) - (40 \sigma_1^2 + 15 \sigma_{11}))\\
& \times \left(216 \sigma_1^3 \sigma_{11}^2 \sigma_{111} + 108 \sigma_1 \sigma_{11}^3 \sigma_{111} + 108 \sigma_1^4 \sigma_{111}^2 - 486 \sigma_1^2 \sigma_{11} \sigma_{111}^2 + 729 \sigma_1 \sigma_{111}^3 \right) \\
= & 2808 \sigma_1^5 \sigma_{11}^2 \sigma_{111} + 1404 \sigma_1^3 \sigma_{11}^3 \sigma_{111} + 1404 \sigma_1^6 \sigma_{111}^2 - 6318 \sigma_1^4 \sigma_{11} \sigma_{111}^2 + 9477 \sigma_1^3 \sigma_{111}^3 \\
&- 3024 \sigma_1^3 \sigma_{11}^3 \sigma_{111} - 1512 \sigma_1 \sigma_{11}^4 \sigma_{111} - 1512 \sigma_1^4 \sigma_{11} \sigma_{111}^2 + 6804 \sigma_1^2 \sigma_{11}^2 \sigma_{111}^2 - 10206 \sigma_1 \sigma_{11} \sigma_{111}^3 \\
= & 13041.
\end{align*}
\end{proof}

Combining Theorem~\ref{ThmChernClassesAsNumbers} and Theorem~\ref{3.18}, we finally get
\begin{theorem}\label{ThmBettiNumbers}
    The Betti numbers of the surface $\Sigma$ are as follows

 (i) $b_0(\Sigma) = b_4(\Sigma) = 1$. 
 
    (ii) $b_1(\Sigma) = b_3(\Sigma) = 42$. 
    
    (iii) $b_2(\Sigma) = 13123$.
\end{theorem}

Similarly, combining Theorem~\ref{ThmChernClassesAsNumbers} and Theorem~\ref{ThmHodgeAndChern}, we get
\begin{theorem}\label{ThmHodgeNumbers}
$h^{2, 0}(\Sigma) = 3233$.
$h^{1, 1}(\Sigma) = 6657$.
\end{theorem}

\subsection{Proof of Theorem~\ref{ThmVHSisMaximal}}

We restate Theorem~\ref{ThmVHSisMaximal} for clarity:

\begin{theorem}\label{ThmVHSisMaximalRestated}
Let $B$ be the moduli space of cubic $5$-folds $Y \subset \mathbb{P}^6$ containing a fixed general cubic $4$-fold $Y_4 \subset \mathbb{P}^6$. For each $b \in B$, let $\Sigma_b$ denote the smooth surface parametrizing planes in $Y_b$. Then the variation of degree $1$ Hodge structures of the family $\{\Sigma_b\}_{b \in B}$ is maximal.
\end{theorem}

\begin{proof}
We will prove that the variation of Hodge structures of weight one associated with the family of surfaces $\{\Sigma_b\}_{b \in B}$ is maximal, i.e., the period map is locally an immersion.

Our strategy involves the following steps:

\begin{enumerate}
    \item Each surface $\Sigma_b$ embeds as a Lagrangian submanifold in a hyper-Kähler fourfold $X$, specifically the Beauville-Donagi variety associated with $Y_4$.
    \item We apply a result from Bai~\cite[Proposition 2]{Bai}, which provides criteria under which the variation of Hodge structures of a Lagrangian family is maximal.
    \item We check that the necessary conditions of Bai's theorem are satisfied in our case.
\end{enumerate}
\textbf{Step 1: Embedding of $\Sigma_b$ into $X$ as a Lagrangian Submanifold}

Let $Y_4 \subset \mathbb{P}^5$ be a general cubic fourfold. As shown by Beauville and Donagi~\cite{BD}, the Fano variety of lines $F_1(Y_4)$ is a hyper-Kähler fourfold $X$. For each cubic $5$-fold $Y_b$ containing $Y_4$, the intersection $Y_b \cap \mathbb{P}^5$ equals $Y_4$. The planes in $Y_b$ intersect $\mathbb{P}^5$ along lines in $Y_4$. Therefore, there is a natural map:
\[
i_b: \Sigma_b \to X,
\]
sending each plane in $Y_b$ to the line in $Y_4$ given by its intersection with $\mathbb{P}^5$.

Iliev and Manivel~\cite[Proposition 4]{IlievManivel} showed that for general $Y_b$, this map is an embedding and $\Sigma_b$ is a Lagrangian submanifold of $X$. Thus, $\{\Sigma_b\}_{b \in B}$ forms a family of Lagrangian submanifolds in $X$.\\
\textbf{Step 2: A Theorem of Lagrangian Families}

We now apply the following result~\cite[Proposition 2]{Bai}:

\begin{theorem}[Bai~\cite{Bai}]\label{ThmBai}
Let $p: \mathcal{L} \to B$ be a family of Lagrangian submanifolds in a hyper-Kähler manifold $X$ of dimension $2n$. Suppose that:
\begin{enumerate}
    \item $X$ is very general,
    \item The transcendental part of $H^2(X, \mathbb{Q})$ satisfies $b_2(X)_{\mathrm{tr}} \geq 5$,
    \item For a fiber $L$ in the family $\mathcal{L}$, we have $h^{1, 0}(L) < 2^{\left\lfloor \frac{b_2(X)_{\mathrm{tr}} - 3}{2} \right\rfloor}$.
\end{enumerate}
Then the variation of Hodge structures of degree $1$ of the family $\mathcal{L}$ is maximal.
\end{theorem}

Our goal is to verify that the conditions of Bai's theorem are satisfied for our family $\{\Sigma_b\}_{b \in B}$.\\
\textbf{Step 3: Verification of the Conditions}\\
\emph{Condition 1: $X$ is very general}

Since $Y_4$ is a general cubic fourfold, the associated hyper-Kähler manifold $X = F_1(Y_4)$ is also very general.\\
\emph{Condition 2: $b_2(X)_{\mathrm{tr}} \geq 5$}

This is due to Corollary~\ref{CorTranscendental} and the fact~\cite[Corollary 2.12]{HuybrechtsCubic} that the $4$-th Betti number of a cubic fourfold $Y_4$ is $23$. 
Thus,
\[
b_2(X)_{\mathrm{tr}} = 23 - 1 = 22 \geq 5.
\]
\emph{Condition 3: $h^{1, 0}(\Sigma_b) < 2^{\left\lfloor \frac{b_2(X)_{\mathrm{tr}} - 3}{2} \right\rfloor}$}

From our earlier computations (Theorem~\ref{ThmHodgeNumbers}), we have:
\[
h^{1, 0}(\Sigma_b) = 21.
\]
Calculating the bound:
\[
2^{\left\lfloor \frac{b_2(X)_{\mathrm{tr}} - 3}{2} \right\rfloor} = 2^{\left\lfloor \frac{22 - 3}{2} \right\rfloor} = 2^{\left\lfloor 9.5 \right\rfloor} = 2^9 = 512.
\]
Since $21 < 512$, the condition is satisfied.

All the conditions of Bai's theorem are satisfied for our family $\{\Sigma_b\}_{b \in B}$. Therefore, we conclude that the variation of Hodge structures of degree $1$ of this family is maximal.

\end{proof}

\end{document}